\documentclass{amsart}

\usepackage{amsmath}
\usepackage{amsthm}

\usepackage{xcolor}
\usepackage{hyperref}
\usepackage{amssymb}

\theoremstyle{plain}

\newtheorem{thm}{Theorem}
\newtheorem{lem}[thm]{Lemma}
\newtheorem{prop}[thm]{Proposition}
\newtheorem{cor}[thm]{Corollary}

\theoremstyle{definition}

\newtheorem{defn}[thm]{Definition}
\newtheorem*{CH}{Continuum Hypothesis}
\newtheorem{exmp}{Example}

\theoremstyle{remark}

\newtheorem{rem}{Remark}

\usepackage{graphicx}

\renewcommand{\O}{\emptyset}

\newcommand{\R}{\mathbb{R}}

\title{Counter-example to continuity of measure in uncountable unions}
\author[S.~Bilkhu]{Simranjeet Bilkhu$^1$}
\author[N.~Forman]{Noah Mills Forman$^1$}
\address{$^1$ Department of Mathematics \& Statistics\\ McMaster University\\ Hamilton, Ontario, L8S 4K1 \\ Canada}
\date{\today}

\thanks{This research is partially supported by NSERC grant RGPIN-2020-06907}

\begin{document}

\begin{abstract}
    Continuity of measure asserts that the measure of the union of an increasing sequence of sets is equal to the supremum of the measures of those sets. We provide counter examples in the case of uncountable unions. We construct the first counter example on the ordinal numbers, and we show that counterexamples also exist in $\R$ if we assume the continuum hypothesis. 
\end{abstract}

\maketitle

\section{Introduction}

\newcommand{\cF}{\mathcal{F}}
\newcommand{\cO}{\mathcal{O}}
\renewcommand{\Pr}{\mathbf{P}}
\newcommand{\Leb}{\textnormal{Leb}}

\emph{Continuity of measure} is the property that, in a measure space $(S,\mathcal{F},\mu)$, given a sequence $(A_n,\,n\ge1)$ of $\cF$-measurable sets with $A_n\subseteq A_{n+1}$ for all $n$,
\begin{equation*}
 \mu\left(\bigcup_{n=1}^\infty A_n\right) = \sup_n \mu(A_n) = \lim_{n\to\infty}\mu(A_n), \tag{CM}\label{eq:CM}
\end{equation*}
and correspondingly for the intersection of a decreasing sequence of sets of finite measure \cite[Lemma 1.15]{Kallenberg}. This is an easy consequence of countable additivity.

Continuity of measure has numerous applications throughout the foundations of probability theory. For example, we can use it to show that for any real-valued random variable $X$, the cumulative distribution function $F_X(n) = \Pr\{X\le n\}$ converges to 1 as $n$ increases and to 0 as $n$ decreases.

In this article, we show that this principle does not, in general, extend to uncountable unions or intersections. To motivate this, consider the following flawed proof that continuity of measure \emph{should} apply to uncountable unions.

\begin{defn}
 Given a measure space $(S,\mathcal{F},\mu)$ and a totally ordered index set $(\mathcal{O},\preceq)$, an \emph{increasing $\cO$-family of $\cF$-sets} is a collection $(A_t)_{t\in \mathcal{O}}$ of $\cF$-measurable sets, with the property that $A_s\subseteq A_t$ whenever $s\preceq t$.
\end{defn}

\begin{proof}[Bogus proof] Let $(\cO,\preceq)$ and $(A_t)_{t\in \mathcal{O}}$ be as above. Let $t_1,t_2,\ldots$ denote any unbounded, increasing sequence in $\cO$. Then because this sequence is unbounded, for every $t\in\cO$ there exists some $n$ such that $t\preceq t_n$, so $A_t\subseteq A_{t_n}$ and $\mu(A_t) \le \mu(A_{t_n})$. From this, we see
\begin{align}
 \bigcup_{t\in\cO} A_t &= \bigcup_{n=1}^\infty A_{t_n} \quad \text{and}\label{eq:bad_pf:union}\\
 \sup_{t\in\cO}\mu(A_t) &= \sup_{n\ge1}\mu(A_{t_n}).\label{eq:bad_pf:sup}
\end{align}
We conclude that
\begin{equation*}
 \mu\left(\bigcup_{t\in\cO} A_t\right) = \mu\left(\bigcup_{n=1}^\infty A_{t_n}\right)
 	= \sup_{n\ge1}\mu(A_{t_n})
 	= \sup_{t\in\cO}\mu(A_t),
\end{equation*}
by \eqref{eq:bad_pf:union}, \eqref{eq:CM}, and \eqref{eq:bad_pf:sup}, respectively.
\end{proof}

If the index set $(\cO,\preceq)$ is $(\R,\leq)$, or any subset of $\R$ under the standard ordering, then the above proof is valid. However, we will show that in general, the proof fails on the grounds that \emph{there may not exist an unbounded, increasing sequence in $\cO$}.

\begin{thm}
\label{thm:general}
 There exists a probability space $(\Omega,\cF,\Pr)$, a totally ordered set $(O,\preceq)$, and an increasing $\cO$-family of $\cF$-sets $(A_t)_{t\in\cO}$ such that
 \begin{equation}\label{eq:general}
  \Pr\!\left(\bigcup\nolimits_{t\in\mathcal{O}} A_t\right) > \sup_{t\in \cO}\Pr(A_t).
 \end{equation}
\end{thm}

To make this more concrete, we can translate our result to the setting of Lebesgue measure, $((0,1),\mathcal{B}((0,1)),\Leb)$, with the help of the following.

\begin{CH}
 If $S$ is an uncountably infinite set then there exists an injection $g\colon \R\to S$. Equivalently, the cardinality of $\R$ is the second-smallest infinite cardinality, denoted
 \begin{equation*}
  |\R| = \aleph_1. \tag{\rm CH}\label{eq:CH}
 \end{equation*}
\end{CH}

The Continuum Hypothesis is famously independent of the standard ZFC axioms of set theory, meaning it can neither be proved true nor false from those axioms \cite{Cohen63}. Intriguingly, there remains ongoing mathematical and philosophical research into whether \eqref{eq:CH} or its negation \emph{should} be adopted as an additional axiom of set theory \cite{CH}.

\begin{thm}
\label{thm:real}
 Assuming \eqref{eq:CH}, there exists a total ordering $\preceq$ on $\cO := (0,1)$, and an increasing $\cO$-family $(A_t)_{t\in\cO}$ of Borel subsets of $(0,1)$, such that
 \begin{equation}\label{eq:real}
  \Leb\!\left(\bigcup\nolimits_{t\in (0,1)} A_t\right) > \sup_{t\in (0,1)}\Leb(A_t).
 \end{equation}
\end{thm}

In fact, we can strengthen both results.

\begin{cor}\label{cor:set_masses:general}
 Fix a totally ordered index set $(\cO,\preceq)$ that realizes the claim of Theorem \ref{thm:general}, in particular with the right-hand side of \eqref{eq:general} equalling 0. Let $(m_t)_{t\in\cO}$ denote a non-decreasing family of non-negative real numbers, and fix $M \ge \sup_{t\in\cO} m_t$. Then there exists a measure space $(\Omega,\cF,\mu)$ and an increasing $\cO$-family of $\cF$-sets $(B_t)_{t\in\cO}$ such that:
 \begin{equation}\label{eq:set_masses:general}
  \forall t\in \cO,\ \mu(B_t) = m_t,\quad \text{and}\quad 
  \mu\!\left(\bigcup\nolimits_{t\in \cO} B_t\right) = M.
 \end{equation}
\end{cor}

\begin{cor}\label{cor:set_masses:real}
 In the setting of Theorem \ref{thm:real}, let $(m_t)_{t\in\cO}$ denote a non-decreasing family of non-negative real numbers, and fix $M \ge \sup_{t\in\cO} m_t$. Then there exists an increasing $\cO$-family $(B_t)_{t\in\cO}$ Borel subsets of $(0,M)$ such that:
 \begin{equation}\label{eq:set_masses:real}
  \forall t\in (0,1),\ \Leb(B_t) = m_t,\quad \text{and}\quad 
  \Leb\!\left(\bigcup\nolimits_{t\in (0,1)} B_t\right) = M.
 \end{equation}
\end{cor}

In Section \ref{sec:bkg}, we will introduce the ordinal numbers, which we then use in Section \ref{sec:results} to prove our claims.

\section{Background}
\label{sec:bkg}

For our audience of probabilists, we introduce the ordinals here, referring the reader to \cite{jech2003set} for proofs of these assertions and further discussion.

\begin{defn}[Well Ordering] A \textit{well order} $\prec$ on a set $A$ is a total ordering in which each subset of $A$ contains a least element:
\begin{equation}
 \forall S\subseteq A,\ \exists x\in S\text{ s.t.\ }\forall y\in S,\ x\preceq y.
\end{equation}
\end{defn}

\begin{exmp}
 $\left( [0, 1], < \right) $ is not a well ordering, since $(\frac{1}{3}, \frac{1}{2}) \subseteq [0, 1]$ does not contain a least element.
\end{exmp}

\begin{defn}[transitive set; ordinal number]
\label{def:ordinal}
 A set of sets $S$ is \textit{transitive} if every element of $S$ is a subset of $S$. A set of sets is an \textit{ordinal number} if it is transitive and is well-ordered by the membership relation  $\in$.
\end{defn}

\begin{exmp}
\label{eg:ordinals}
    The following sets are all ordinals. In the set-theoretic foundations of mathematics, it is conventional to formally define the non-negative integers as these ordinals, as follows: 
    \begin{itemize}
        \item $0 := \O$,
        \item $1 := \left\{ \O \right\} = \{0\} $,
        \item $2 := \left\{ \O,  \left\{ \O \right\} \right\} = \{0,1\} $,
        \item $3 := \left\{ \O, \left\{ \O \right\}, \left\{ \O, \left\{ \O \right\}  \right\} \right\}  = \{0,1,2\}$.
    \end{itemize}
    We denote the set of all finite ordinals by
    \begin{equation}
     \omega_0 := \big\{ \O, \{\O\}, \big\{\O,\{\O\}\big\}, \ldots \big\} = \{0,1,2,\ldots\} = \mathbb{N}.
    \end{equation}
    This is the least infinite ordinal.
\end{exmp}

It is not obvious from Definition \ref{def:ordinal} that the sets in Example \ref{eg:ordinals} are the unique ordinals of cardinality 0, 1, 2, and 3, but this is indeed the case. See \cite[Lemma 2.11]{jech2003set}.

\begin{rem}
 Order-preserving bijections are the ``isomorphisms'' for the well-ordered sets. The ordinal numbers are a canonical choice of representatives \cite[Thm 2.12]{jech2003set} of the resulting equivalence classes, called ``order types.''
\end{rem}


Because the ordinals are transitive sets, for ordinals $\alpha$ and $\beta$,
$$\alpha \in\beta\text{ if and only if }\alpha\subset\beta.$$

\begin{defn}
Let $\alpha$ and $\beta$ be ordinals. We say that
\begin{center}
    $\alpha \prec \beta$ if and only if $\alpha \in \beta$, or equivalently, if and only if $\alpha \subset \beta$. 
\end{center}
\end{defn}









\newcommand{\Ord}{\textit{Ord}}

We denote the class of all ordinals by $\Ord$. 

\begin{lem}[Lemma 2.11 and the subsequent remarks in \cite{jech2003set}] \label{lem:members-are-ordinals}
 For each $\alpha\in\Ord$,
 \begin{equation}
 \label{eq:members-are-ordinals}
  \alpha = \{\beta\in\Ord\colon \beta\prec\alpha\}.
 \end{equation}
 Moreover, $\prec$ is a well-ordering on $\Ord$.
\end{lem}


\begin{defn}
 The \emph{union set} of a set of sets $A$ is
 \begin{equation}
  \bigcup A := \bigcup_{\zeta \in A} \zeta.
 \end{equation}
\end{defn}


Given an ordinal $\alpha$, the sets
$$\alpha \cup \{\alpha\} \qquad \text{and}\qquad \bigcup \alpha$$
are also ordinals \cite[p.\ 20]{jech2003set}.

\begin{defn}[Successor and Limit Ordinals]
 The \emph{successor ordinal} to $\alpha$ is the set $\alpha+1 := \alpha \cup \{\alpha\} $. If an ordinal $\xi$ not a successor ordinal, it is called a \emph{limit ordinal}.
\end{defn}

Here is an easily proved alternative characterization of this distinction:
%
\begin{equation}
\label{eq:limit_ordinal}
 \alpha = \begin{cases}
  \bigcup \alpha    & \text{if $\alpha$ is a limit ordinal, or}\\
  \bigcup \alpha + 1    & \text{if $\alpha$ is a successor}.
 \end{cases}
\end{equation}

To emphasize the relationship between ordinals and the Continuum Hypothesis, we note that there exist ordinals of every cardinality \cite[p.\ 29]{jech2003set}. The following theorem is of particular importance to this discussion. Define
\begin{equation}\label{eq: omega_1}
    \omega_1 := \big\{\xi\in \Ord\colon \xi\text{ is countable}\big\}.
\end{equation}

\begin{thm}[Page 30 in \cite{jech2003set}]\label{thm:omega1-least-unctounable}
 The set $\omega_1$ is the least uncountable ordinal. In particular:
\[
    |\omega_1| = \aleph_1.
\]
\end{thm}

\begin{prop}
\label{prop:omega1-not-successor}
 $\omega_1$ is a limit ordinal.
\end{prop}

\begin{proof}
 If $\omega_1$ were a successor ordinal, then because each element of $\omega_1$ is countable, it would be successor to a countable ordinal, and thus it would be countable as well.
\end{proof}

We also require the following language.
\begin{defn}
  Given a set $\Omega$, a set $B\subseteq\Omega$ is \emph{co-countable} with respect to $\Omega$ if $\Omega\setminus B$ is countable.
\end{defn}

\begin{exmp}
\begin{itemize}
    \item $\R - \{1\}$ and $\R - \mathbb{Q}$ are co-countable subsets of $\R$.
    \item No countable subset of $\R$ is co-countable in $\R$.
    \item The set $(0, 1)$ is neither a countable, nor a co-countable subset of $\R$
\end{itemize}
\end{exmp}

\section{Proofs of main results}
\label{sec:results}

\begin{proof}[Proof of Theorem \ref{thm:general}]
 Following the notation of the theorem statement, we take:
 \begin{itemize}
  \item   $\Omega = \omega_1$;
  \item $\mathcal{F} = \{ \text{Countable and co-countable subsets of $\omega_1$}\}$,
  \item $\displaystyle \mathbf{P}( A ) = \begin{cases}
        1 &\text{if $A$ is co-countable,} \\
        0 &\text{if $A$ is countable;}
\end{cases}$
  \item $\cO = \omega_1$ under the usual ordering of the ordinals (via the subset relation);
  \item $A_\xi = \xi$ for every $\xi\in\omega_1$.
 \end{itemize}
 
 To clarify, we are using $\omega_1$ in three ways: as our sample space, as our totally ordered index set $\cO$, and as our increasing $\cO$-family of $\cF$-measurable sets.
 
 Indeed, by the definition of $\omega_1$ in \eqref{eq: omega_1}, all of its elements are countable. And by the transitive set property of the ordinals, elements of $\omega_1$ are also subsets of $\omega_1$; thus, $\omega_1\subset\cF$, and
 $$\sup_{\xi\in\omega_1}\mathbf{P}(\xi) = 0.$$
 Finally, by Proposition \ref{prop:omega1-not-successor} and \eqref{eq:limit_ordinal},
 \[
   \mathbf{P} \left( \bigcup_{\xi \in \omega_1} \xi \right) = \mathbf{P} \left( \omega_1 \right) = 1.\qedhere
 \]
\end{proof}
\medskip

\begin{proof}[Proof of Theorem \ref{thm:real}]
We assume the continuum hypothesis, \eqref{eq:CH}. This implies that the cardinality of $\omega_1$ equals that of $\R$, by Theorem \ref{thm:omega1-least-unctounable}. Under this assumption, there exists a bijection $f: (0, 1) \to \omega_1$. To avoid confusion between inverse and preimage, define:
\[
  F(\beta) = \left\{f ^ {- 1} ( \alpha ) : \alpha \in \beta \right \} \quad \text{ where } f ^ {- 1} \text{ is the inverse map of $f$.} 
\]
Consider the probability space $\big( ( 0, 1), \mathcal{B}( (0, 1) ), \Leb  \big) $. Define  
$$
\mathcal{S} = \Big\{ F(\xi )\ \Big|\ \xi \in \omega_1 \Big\}.
$$ 
Suppose $A, B \in \mathcal{S}$. By definition, there exist $\alpha, \beta \in \omega_1$, such that
$$ A = F(\alpha) = \big\{ f ^{-1}( \phi ) \colon \phi \in \alpha \big\}  \text{\quad and\quad}
B = F(\beta) = \big\{ f ^{-1} ( \phi ) \colon \phi \in \beta \big\}.$$
By Lemma \ref{lem:members-are-ordinals}, either $\alpha \subseteq \beta$ or $\beta \subseteq \alpha$. Thus, either $A \subseteq B$ or $B \subseteq A$. Moreover, each element in  $\mathcal{S}$ is countable, since each $\xi \in \omega_1$ is countable. As a result:
\[
\sup_{A \in \mathcal{S}} \Leb( A ) = 0.
\]
Then
\begin{align*}
    \bigcup_{A \in \mathcal{S}} A 
    &= \bigcup_{\xi \in \omega_1} F(\xi) \\
    &= \bigcup_{\xi \in \omega_1} \left\{ f^{-1}(\alpha) : \alpha \prec \xi \right\} \\
    &= \bigcup_{\xi \in \omega_1} \left\{ t \in (0, 1) : f(t) \prec \xi \right\} \\
    &= \left\{ t \in (0, 1) : \exists \xi \in \omega_1 \text{ such that } f(t) \prec \xi \right\} \\
    &= (0, 1) \qquad \text{by Proposition \ref{prop:omega1-not-successor}}.
\end{align*}
Thus, 
\[
\Leb\left(\bigcup_{A \in \mathcal{S}} A\right) = 1 \neq 0 = \sup_{A \in \mathcal{S}} \Leb( S ). \qedhere
\]
\end{proof}
\medskip

\begin{proof}[Proof of Corollary \ref{cor:set_masses:general}]
Fix:
\begin{itemize}
 \item a totally ordered index set $\left( \mathcal{O}, \preceq \right)$,
 \item a probability space $\left( \Omega, \mathcal{F}, \mathbf{P} \right)$,
 \item and an increasing $\mathcal{O}$ family $\left( A_t \right) _{t \in \mathcal{O}}$ of $\mathcal{F}$ sets
\end{itemize}
that realize the claim of Theorem \ref{thm:general}, with $\sup_{t\in\mathcal{O}}\Pr(A_t) = 0$. Let $\left( m_t \right) _{t \in \mathcal{O}}$ be a bounded, non-decreasing family of non-negative real numbers, and fix $M \ge  \sup_{t \in \mathcal{O}}m_t$.

Our strategy will be to define a space $(E,\mathcal{E},\mu)$ as a disjoint union $E = \Omega\sqcup \R_+$ and define $\mu$ as the sum of a rescaled $\Pr$ on the $\Omega$ component, plus Lebesgue measure on the real component. To formalize this, we define:
\begin{itemize}
    \item  $E = \left( \Omega \times \left\{ 1 \right\}  \right) \cup \left( \R_+ \times \left\{ 2 \right\}  \right) $,
    \item $\mathcal{E} = \left\{ \left( S \times \left\{ 1 \right\}  \right) \cup \left( T \times \left\{ 2 \right\}  \right): S \in \mathcal{F}, T \in \mathcal{B}\left( \R_+ \right)   \right\} $, and
    \item $\mu: \mathcal{E} \to \R_+$ given by
    \begin{align*}
     \mu\left( S \times \left\{ 1 \right\} \cup T \times \left\{ 2 \right\}  \right) &= \frac{s}{p}\mathbf{P}\left( S \right) + \Leb \left( T \right)\text{, where}\\
     s &:= M - \sup_{t\in\cO}m_t,\\
     p &:= \Pr\left(\bigcup_{t\in\cO} A_t\right).
    \end{align*}
\end{itemize}

Let
\[
  B_t = \Big( A_t\times\{1\}\Big) \cup \Big( (0,m_t)\times\{2\}\Big).
\]
Then
\[
  \mu(B_t) = \Leb( 0, m_t) + \frac{s}{p} \mathbf{P}\left( A_t \right)  = m_t,
\] 
and
\begin{align*}
 \mu\left( \bigcup_{t \in \mathcal{O}} B_t \right) &=  \mu \left( \bigcup_{t \in \mathcal{O}}( 0, m_t) \times \{ 1 \}  \cup \bigcup_{t \in \mathcal{O}} A_t\times \{ 2 \}   \right) \\
 &=  \Leb\!\left( \bigcup_{t \in \mathcal{O}}( 0, m_t)  \right) + \frac{s}{p} \mathbf{P}\left( \bigcup_{t \in \mathcal{O}} A_t \right) \\
 &=  \sup_{t \in \mathcal{O}} m_t + \frac{s}{p}p \\
 &= M. \qedhere
\end{align*}
\end{proof}
\medskip

\begin{proof}[Proof of Corollary \ref{cor:set_masses:real}]
Fix $( m_t) _{t \in \mathcal{O}}$ a non-decreasing family of non-negative real numbers, and fix $M \ge \sup_{t \in \mathcal{O}} m_t$. As per the proof of theorem 3, there exists increasing $\mathcal{O}$-family  $( A_t ) _{t \in \mathcal{O}}$ of Borel subsets of $( 0, 1) $, such that
\[
1 = \Leb\!\left( \bigcup_{t \in \left( 0, 1 \right) } A_t \right) > \sup_{t \in \left( 0, 1 \right) } Leb\left( A_t \right) = 0
.\] 
Let $s = M - \sup_{t \in \mathcal{O}} m_t$. Define
\begin{gather*}
 D_t := s A_t = \{sa\colon a\in A_t\}\text{, so}\\
 s = \Leb\!\left( \bigcup_{t \in \left( 0, 1 \right) }  D_t \right) > \sup_{t \in \mathcal{O}} \Leb \left( D_t  \right)  = 0.
\end{gather*}
Now, let
\begin{gather*}
  B_t := D_t \cup \left( s, m_t + s \right)\text{ so that}\\
  \Leb( B_t ) = \Leb\big( ( s, m_t + s ) \big) + \Leb( D_t ) = m_t + 0.
\end{gather*}
Then
\begin{align*}
    \Leb\!\left( \bigcup_{t \in \left( 0, 1 \right) } B_t  \right) &=  \Leb\!\left( \bigcup_{t \in \left( 0, 1 \right) } A_t \right) + \Leb\!\left( \bigcup_{t \in \left( 0, 1 \right) } \left( s, m_t + s \right) \right) \\
    &= s + \sup_{t \in \mathcal{O}} m_t \\
    &= M. \qedhere
\end{align*}
\end{proof}
\bibliographystyle{plain}
\bibliography{refs(5).bib}

\end{document}